\documentclass{article}
\usepackage{amsmath,amssymb,amsfonts,amsthm,mathrsfs,color,graphics}
\usepackage[all]{xy}
\usepackage[dvips]{graphicx}
\usepackage{tikz} 
\usepackage[a4paper]{geometry}
\usepackage{textcomp}
\newtheorem{prop}{Proposition}[section]

\newtheorem*{theorem*}{Theorem}
\newtheorem{theorem}{Theorem}

\theoremstyle{remark}
\newtheorem{remark}[prop]{{Remark}}
\theoremstyle{remark}
\newtheorem{example}[prop]{Example}
\theoremstyle{definition}
\newtheorem{definition}[prop]{Definition}

\newtheorem{lemma}[prop]{Lemma}

\newcounter{exercise}[section]

\newcommand{\mb}{\mathbb}

\newcommand{\ra}{\rightarrow}

\newcommand{\hra}{\hookrightarrow}

\newcommand{\te}[1]{\text{\textnormal{#1}}}
\newcommand{\pc}[2]{ \langle #1 \rangle_{\te{pc}}^{#2}}

\newcommand{\zr}[2]{ \langle #1 \rangle_{#2}}
\newcommand{\cpt}[1]{ \langle #1 \rangle_{\te{cpt}}}

\newcommand{\scr}[1]{\mathscr{#1}}
\newtheorem{art}[prop]{}

\newcommand{\kcirc}{{K^\circ}}
\newcommand{\ktilde}{{\widetilde{K}}}

\newcommand{\Ccal}{{\mathscr C}}
\newcommand{\Wcal}{{\mathscr W}}
\newcommand{\Kcal}{{\mathscr K}}

\newcommand{\Xcal}{{\mathscr X}}
\newcommand{\Ical}{{\mathscr I}}
\newcommand{\Jcal}{{\mathscr J}}
\newcommand{\Ycal}{{\mathscr Y}}
\newcommand{\Zcal}{{\mathscr Z}}
\newcommand{\Ucal}{{\mathscr U}}

\newcommand{\Vcal}{{\mathscr V}}

\newcommand{\Spec}{{\rm Spec}}

\newcommand{\Ocal}{{\mathscr O}}

\newcommand{\tdop}{{\mathbb T}}

\newcommand{\pdop}{{\mathbb P}}

\newcommand{\cdop}{{\mathbb C}}

\newcommand{\qdop}{{\mathbb Q}}

\newcommand{\rdop}{{\mathbb R}}

\newcommand{\adop}{{\mathbb A}}
\newcommand{\zdop}{{\mathbb Z}}

\newcommand{\y}{{\bf{y}}}

\begin{document}
\title{Nagata's compactification theorem for normal toric varieties \\ over a valuation ring of rank one}
\author{Alejandro Soto}
\maketitle

\begin{abstract}
We prove, using invariant Zariski--Riemann spaces, that every normal toric variety over a valuation ring of rank one can be embedded as an open dense subset into a proper toric variety equivariantly. This extends a well known theorem of Sumihiro for toric varieties over a field to this more general setting. 
\end{abstract}

\section{Introduction}

Toric geometry has been an important subject in algebraic geometry since its very beginnings, one of the reasons being the fact that its combinatorial aspects allow very concrete geometric manipulations. This leads to many explicit examples and constructions in algebraic geometry.

As every normal toric variety over a field is constructed from a fan, many geometric properties can be understood in  combinatorial terms. For instance, a normal toric variety over a field is proper if and only if the associated fan is complete. Furthermore, by modifying the fan we can obtain a modification of the given variety. One of the most important examples of this phenomenon is the normalized blow up of a toric variety along a center that is invariant under the action of the torus.  

Toric schemes over arbitrary valuation rings of rank one have been introduced by Gubler in \cite{gubler12} in order to generalize tropical compactifications of closed subvarieties of the torus $\mb{G}^n_{\te{m},K}$ over an arbitrary valued field $K$ of rank one. Those schemes generalize the toric schemes over discrete valuation rings studied by Mumford in the 70's, see \cite[Ch. IV \S 3]{mumford73}. In \cite{gubler_soto13}, Gubler and the author have generalized the classification of normal toric varieties over fields given by rational fans to the setting of normal toric schemes of finite type over an arbitrary valuation ring of rank one. The classification is given in terms of certain admissible fans in $\mb{R}^n\times \mb{R}_+$, where the extra factor  $\mb{R}_+$ takes into account the valuation of the ground ring. See \S 2 and \cite{gubler_soto13} for details.

The combinatorial aspects of these toric schemes extend in a natural way the classical theory over a field. To be more precise let us fix a rank one valued field $K$ with valuation ring $\kcirc$ and consider the split torus $\tdop:=\mb{G}_{\te{m},\kcirc}^n$ over $\kcirc$. A $\tdop$-toric variety $\Ycal$ over $\kcirc$ is a flat integral separated scheme of finite type over $\kcirc$ such that its generic fiber contains $T:=\mb{G}_{\te{m},K}^n$ as an open dense subset and the multiplication action of $T$ on itself extends to an algebraic action of  $\tdop$ on $\Ycal$ over $\kcirc$.  Suppose that $\Ycal$ is an affine normal $\tdop$-toric variety over $\kcirc$ and let  $\sigma \subset \mb{R}^n\times \mb{R}_+$ be its corresponding admissible cone. The generic fiber $\Ycal_\eta$ is a $T$-toric variety over $K$ described by the recession cone of the polyhedron $\sigma_1:=\{w\in \rdop^n|(w,1)\in \sigma\}$ and the torus orbits in the special fiber correspond to the vertices of  $\sigma_1$. In this way, we obtain a complete description of the torus orbits in the generic and the special fiber in terms of the structure of the admissible fan. 

It is natural to ask which other properties of toric varieties can be extended to  the setting of $\tdop$-toric varieties over $\kcirc$. One of the main difficulties when trying to generalize them is the fact that, unless the valuation is discrete, we are working in a non-noetherian setting. Hence many standard results in algebraic geometry cannot be applied immediately. Regardless of the absence of the noetherian condition, the underlying topological space of a $\tdop$-toric variety is a noetherian topological space. Furthermore, the generic fiber is a $T$-toric variety over $K$ and the special fiber is a separated scheme of finite type over the residue field $\ktilde$. The associated reduced scheme of every irreducible component of the latter is a toric variety over $\ktilde$, see \S 2.   

In this paper, we address the question of whether a normal $\tdop$-toric variety over $\kcirc$ can be embedded into a proper $\tdop$-toric variety over $\kcirc$. Our  main result answers this question affirmatively. This generalizes a well known theorem of Sumihiro on the equivariant completion of normal toric varieties, see \cite{sumihiro1} and \cite{sumihiro2}. More precisely, we have the following statement.

\begin{theorem}\label{theorem1}
Let $\Ycal$ be a normal $\tdop$-toric variety over the valuation ring $\kcirc$. Then there exists an equivariant open immersion $\Ycal \hra \Ycal_{\te{cpt}}$ into a proper $\tdop$-toric variety $\Ycal_{\te{cpt}}$ over $\kcirc$.
\end{theorem} 

The proof of this theorem follows the lines of the proof of Nagata's compactification theorem as presented by Fujiwara--Kato in \cite[Appendix F]{fujiwara-kato14}. The main tool used there is the Zariski--Riemann space associated to a pair $(\Ycal, \Ucal)$, where $\Ucal \subset \Ycal$ is a quasi-compact open subscheme of $\Ycal$. It is defined as 
\[ \zr{\Ycal}{\Ucal}:=\varprojlim \Ycal_i,\]
where the limit is taken over the collection of  $\Ucal$-admissible blow ups over $\Ycal$, i.e. blow ups with center disjoint from $\Ucal$. Note that we can dentify $\Ucal$ with an open subset of the Zariski--Riemann space $\zr{\Ycal}{\Ucal}$.  We remark that in the case  $\Ucal=\emptyset$, these spaces did play a key role in the first proof given by Sumihiro in \cite{sumihiro1}.

In our setting $\Ycal$ is a normal $\tdop$-toric variety over the valuation ring $\kcirc$, $\Ucal$ is an open invariant subscheme and the transition maps are equivariant, see Definition \ref{admissible} and \ref{def:zariski-riemann} below. 
In this case we get a locally ringed space endowed with an action of the torus $\tdop$. 

For toric varieties over a field, there are purely combinatorial proofs of the existence of the equivariant completion. More precisely, it has been proved that every rational fan can be completed, see for instance \cite[III. Theorem 2.8]{ewald}, \cite{ewaldishida} and \cite{rohrer11}. This gives rise to an equivariant open embedding of the original variety into a complete normal toric variety. We point out that, as the toric schemes over discrete valuation rings are described combinatorially by rational fans, these results also provide an equivariant completion in the discretely valued case. 

The structure of the paper is as follows. In section \S 2 we recall the basic definitions, constructions and examples of $\tdop$-toric varieties over rank one valuation rings. We have a new result in this section, Proposition \ref{normalization}, where we prove that the normalization of a projective $\tdop$-toric variety with a linear action of the torus can be constructed in a canonical way, extending the classical results over a field and over a discrete valuation ring.  This construction is done using the combinatorial description of the $\tdop$-toric variety given by the subdivided weighted polytope, see \S \ref{projective}. Although this result is not needed for the proof of Theorem \ref{theorem1}, we include it here for completeness of the presentation of the $\tdop$-toric varieties over $\kcirc$.

In section \S3 we consider a $\tdop$-toric variety $\Ycal$ over the valuation ring $\kcirc$ and an open invariant subscheme $\Ucal$. We  give a detailed description of the $\Ucal$-admissible blow ups. We show that they are preserved under composition and that the collection of all $\Ucal$-admissible blow ups is filtered. This allows us to define the invariant Zariski--Riemann space associated to $(\Ycal, \Ucal)$ as the projective limit over all the $\Ucal$-admissible blow ups.

 Finally in \S 4 we give a proof of our main result. We proceed as follows: we consider an open invariant affine covering $\{ \Ucal_i \}$ of our normal $\tdop$-toric variety $\Ycal$. For each open $\Ucal_i$, we take a compactification $\overline{\Ucal_i}$ and with it  we construct a locally ringed space $\pc{\Ucal_i}{\Ycal}$. It has the property that if $\Ucal_i \subset \Ucal_j$ then $\pc{\Ucal_i}{\Ycal} \subset \pc{\Ucal_j}{\Ycal}$. 
From this invariant covering, we get a collection of invariant locally ringed spaces $\{ \pc{\Ucal_i}{\Ycal}\}$. Due to the compatibilty with respect to the inclusions, we can glue these spaces along common intersections in order to get a $\tdop$-invariant locally ringed space $\cpt{\Ycal}$ called the Zariski--Riemann  compactification of $\Ycal$. By construction, we have $\Ycal \subset \langle \Ycal \rangle_{\te{cpt}}$. 
Finally in Proposition \ref{algebraic}, we show that the locally ringed space $\cpt{\Ycal}$  is algebraic in the following sense: there exists a scheme ${\Ycal}_{\te{cpt}}$ over $\kcirc$ which contains $\Ycal$ as an open and dense subset and such that the Zariski--Riemann space associated to $({\Ycal}_{\te{cpt}},{\Ycal})$ is isomorphic to $\cpt{\Ycal}$.  This scheme is in fact a proper $\tdop$-toric variety over the valuation ring $\kcirc$, which concludes the proof of the Theorem 1. 
 
\vspace{0.5cm}

{\small 
The author would like to thank to Walter Gubler for suggesting the original problem and to Lorenzo Fantini for many helpful remarks on a preliminary version of this paper. I also would like to thank to the KU Leuven  and to the Goethe Universit\"at Frankfurt am Main for the warm atmosphere and for the great working conditions. I am grateful to the referee for his comments and suggestions.
}

\begin{center}
{\it Notation}
\end{center}

For sets $A$ and $B$, the notation $A\subset B$ includes the possibility $A=B$. We let $A\backslash B$ denote the complement of $B$ in $A$. The set of non-negative numbers in $\zdop$, $\qdop$ or $\rdop$ is denoted by  $\mb{Z}_+$, $\qdop_+$ or $\mb{R}_+$ respectively. All rings and algebras are commutative with unity. For a ring $A$, the group of units is denoted by $A^\times$. A variety over a field $K$ is an irreducible and reduced scheme which is separated and of finite type over $K$. 

In the whole paper, we fix a {\it valued field} $(K,v)$ which means here that $v$ is a valuation on the field $K$  with value group $\Gamma:=v(K^\times)\subset \mb{R}$. Note that $K$ is not required to be algebraically closed or complete and that its valuation can be trivial. We have a valuation ring $K^\circ:=\{x\in K\mid v(x)\geq 0\}$ with maximal ideal $K^{\circ \circ}:=\{x\in K\mid  v(x)>0\}$ and residue field $\widetilde{K}:=K^\circ / K^{\circ \circ}$. Let $S=\Spec (\kcirc)=\{\eta, s\}$, with $\eta$ its generic point. 

We denote by $M$ a free abelian group of rank $n$ and by $N:=\te{Hom}(M, \zdop)$ its dual. For $G\subset \mb{R}$ an abelian subgroup,  we write $M_G:=M\otimes_\zdop G$ for the base change of $M$ to $G$.

\section{Toric varieties over valuation rings}

We denote by $\tdop =\te{Spec}(\kcirc[M])$ the split torus of rank $n$ over the valuation ring $\kcirc$ and by $T$ its generic fiber.  In this section, we review the main definitions and results of $\tdop$-toric varieties over valuation rings of rank one. For a more detailed exposition, we refer the reader to the papers \cite{gubler12} and \cite{gubler_soto13}. There is one new result in this section, namely Proposition \ref{normalization}, where we extend a well known result concerning the normalization of a projective $\tdop$-toric variety with a linear action of the torus, see \cite[Chapter 5 \S B]{gkz} and \cite[Proposition 2.3.8]{qu}.

\begin{definition}
 A \emph{$\tdop$-toric variety} over $K^\circ$ is an integral scheme $\Ycal$ separated flat of finite type over $K^\circ$ such that the generic fiber $\Ycal_\eta$ contains $T$ as an open dense subset and the multipication action $T\times_K T\to T$ extends to an algebraic action $\tdop \times_{K^ \circ} \Ycal \to \Ycal$ over $\kcirc$. 
\end{definition}

It follows from the definition that the generic fiber is a $T$-toric variety over $K$. The special fiber $\Ycal_s$ is a separated scheme of finite type over the residue field $\widetilde{K}$ of $K$. The induced reduced varieties associated to the irreducible components of $\Ycal_s$ are toric varieties over $\widetilde{K}$. The dense torus acting on each irreducible component may vary, see \cite[Corollary 6.15]{gubler12}. As the scheme $\Ycal$  is flat over $K^\circ$, every  component of the special fiber has the same dimension of the generic fiber. 

If the valuation is trivial, the generic and the special fibers coincide and the definition of a $\tdop$-toric variety agrees with the usual definition of a toric variety over a field. Note that as these schemes are flat and of finite type over $\kcirc$, then they are of finite presentation over $\kcirc$. This follows from \cite[Premi\`ere partie, Corollaire 3.4.7]{raynaud}.

In order to construct some examples, we consider the following $K^\circ$-algebras associated to cones in $N_\rdop \times \rdop_+$. 

\art A subset $\sigma\subset N_\rdop\times \rdop_+$ is called a \emph{$\Gamma$-admissible cone} if it can be written as
\[ \sigma = \bigcap_{\textrm{finite}} \left\{ (w, t)\in N_\rdop \times \rdop_+ | \langle m_i, w \rangle +c_it\geq 0 \right\}, \quad m_i\in M, c_i\in \Gamma, \]
and it does not contain a linear subspace of positive dimension. We denote by $\sigma_r$ the polyhedral complex induced by $\sigma$ in $N_\rdop$ at level $r$, that is $\sigma_r:=\{ w\in N_\rdop | (w,r)\in \sigma\}$. Note that $\sigma_0$ is the recession cone of the polyhedron $\sigma_1$, denoted by $\te{rec}(\sigma_1)$. 

Given a $\Gamma$-admissible cone $\sigma$, we define the following algebra over $\kcirc$
\[ K[M]^ \sigma:= \left\{ \sum a_u\chi^u \in K[M]| \langle u, w \rangle + tv(a_u)\geq 0, \forall (w, t) \in \sigma \right\} .\] 
It is a flat $\kcirc$-algebra, as it is torsion free. If the valuation is discrete it is of finite type over $\kcirc$. If the valuation is not discrete it is of finite type if the vertices of the polyhedron $\sigma_1$ are in $N_\Gamma$, see  \cite[Proposition 6.9]{gubler12}. The algebra $K[M]^\sigma$ is normal and its quotient field is equal to $K(M)$. Note that it is canonically endowed with an $M$-graduation, hence the affine scheme $\text{Spec}(K[M]^\sigma)$ has an algebraic action of the torus $\tdop$ which extends the multiplication action of $T$ on itself. If $K[M]^\sigma$ is finitely generated, it gives rise to a normal $\tdop$-toric variety $\Ycal_\sigma:=\te{Spec}(K[M]^\sigma)$ over $K^\circ$. In this case, the generic fiber $(\Ycal_\sigma)_\eta$  is the toric variety over $K$ associated to the cone $\sigma_0=\te{rec}(\sigma_1)$.  The geometry of the special fiber is controlled by the polyhedron $\sigma_1$, for instance the irreducible components of $(\Ycal_\sigma)_s$ are in bijection with the vertices of $\sigma_1$. Roughly speaking, each component is given by the local cone of a vertex. The character lattice of the torus acting on the irreducible component associated to the vertex $w_i$ is isomorphic to $M_i=\{m\in M|\langle m,w_i\rangle \in \Gamma\}$, see \cite[Corollary 6.15]{gubler12}.

When the valuation is discrete, and $\pi \in K$ is a choice of uniformizing parameter, the algebra $K[M]^\sigma$ is generated by the elements $\{ \pi^k \chi^u\}_{\{(u,k)\in I\}}$, where $I$ is a set of generators of the semigroup $\check{\sigma}\cap (M\times \zdop)$. If the valuation is not discrete, we can give a set of generators of this algebra as follows. Let us consider the set of vertices $\{w_i\}$ of the polyhedron $\sigma_1 \subset N_\rdop$ and let $\{ u_{ij}\}_j$ be the generators of the semigroup $\check{\sigma_i}\cap M$, with  $\sigma_i=\te{LC}_{w_i}(\sigma_1)$ the local cone of $\sigma_1$ at $w_i$. Then, we have that 
\[ K[M]^ \sigma=K^\circ [\alpha_{ij}\chi^{u_{ij}}],\]
where the constants $\alpha_{ij}$ satisfy the conditions $v(\alpha_{ij})+\langle u_{ij}, w_i \rangle =0$. 

\begin{remark}
It follows from \cite[Lemma 6.13]{gubler12} that if the valuation $v$ is not discrete or if the vertices of $\sigma_1$ are contained in $N_\Gamma$, the special fiber of $\Ycal_{\sigma}$ is reduced. In this case,  every irreducible component of $(\Ycal_{\sigma})_s$ is a toric variety over $\ktilde$.
\end{remark}

\begin{example} Suppose the valuation is not discrete and consider the cone $\sigma$ in $\rdop^2 \times \rdop_+$ generated by the polyhedron $\sigma_1 \times \{ 1\}$, where $\sigma_1\subset \mathbb{R}^2$ is the polytope \[\sigma_1=\te{Conv}\{(0,0),(0,\lambda), (\lambda,0)\}\subset \rdop^2, \quad \lambda >0.\]
See Figure \ref{polyhedron}. We assume $\lambda \in \Gamma$.  In this case the algebra is given by
\[ K[M]^\sigma=K^\circ [x,y,ax^{-1},ay^{-1}, ax^{-1}y, ax y^{-1}] \]
with $a\in \kcirc$ such that $v(a)=\lambda$. It is isomorphic to
\[ K^\circ[x,y,ax^{-1}y^{-1}]=K^\circ[x,y,z]/(xyz-a). \]

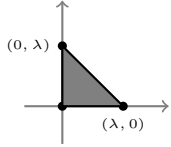
\begin{figure}[!htb]
\[
\begin{xy}
(0,0)*+{
\begin{tikzpicture}
\draw[thick,gray,->] (-.5,0) -- (1.4,0);
\draw[thick,gray,->] (0,-.5) -- (0,1.4);
\fill[black] (0,0) circle (.06cm);
\fill[black] (.8,0) circle (.06cm);
\fill[black] (0,.8) circle (.06cm);
\filldraw [gray] (0,0) -- (0,.8) -- (.80,0) -- (0,0);
\draw [thick, black] (0,0) -- (0,.8) -- (.80,0) -- (0,0);
\draw[black](0.8,0) -- node[below=1pt] {\tiny{$(\lambda,0)$}} (0.8,0);
\draw[black](0,.8) -- node[left=1pt] {\tiny{$(0,\lambda)$}} (0,.8);
\end{tikzpicture}
};
\end{xy} 
\]
\caption{Level 1 of the cone $\sigma$}
\label{polyhedron}
\end{figure}
With this algebra we get a $\tdop$-toric variety $\Ycal_\sigma=\te{Spec}(K[M]^\sigma)$ whose generic fiber is the toric surface given by \[\te{Spec}\left(K[x,y,z]/(xyz-a)\right)\simeq \mb{G}_{\te{m},K}^2\] and the special fiber is the reduced scheme of finite type over $\ktilde$ given by $\te{Spec}(\ktilde[x,y,z]/(xyz))$. Note that each irreducible component is isomorphic to $\mathbb{A}^2_\ktilde$, with its structure as a toric variety over $\ktilde$ with torus $\tdop_\ktilde=\te{Spec}(\ktilde[\zdop^2])$.  If we take $K=\cdop \{ \{ t \} \}$ to be the Puiseaux series and $a=t$, then this example gives a one parameter family of complex tori  degenerating to three copies of the complex affine plane. Note that in this case $\kcirc$ is non-noetherian.
\end{example}

\art This construction can be generalized by gluing affine $\tdop$-toric varieties. For this, we note that faces of $\Gamma$-admissible cones are again $\Gamma$-admissible cones and they give rise to open immersions. In this way given a $\Gamma$-admissible fan $\Sigma$,  i.e. a fan consisting of $\Gamma$-admissible cones, one can glue affine $\tdop$-toric varieties along the open immersions coming from the common faces. This procedure gives rise to a normal $\tdop$-toric variety $\Ycal_\Sigma$ over $\kcirc$. 

In this way we obtain, up to isomorphism, all normal $\tdop$-toric varieties over $K^\circ$. More concretely we have the following theorem which extends the well known classification of normal toric varieties over a field in terms of convex rational polyhedral fans, see  \cite[Chapter I, \S 2 Theorem 6]{mumford73}.

\begin{theorem*}
Let $\Ycal$ be a normal $\tdop$-toric variety over $\kcirc$. Then there is a $\Gamma$-admissible fan $\Sigma$ such that $\Ycal \simeq \Ycal_\Sigma$.  If the valuation is not discrete, the cones in this $\Gamma$-admissible fan satisfy an extra condition, namely the vertices of the corresponding level 1 polyhedron must have coordinates in $\Gamma$.
\end{theorem*}
\begin{proof}
If $\Ycal$ is affine, this follows from \cite[Theorem 1]{gubler_soto13}. If $\Ycal$ is not affine, it follows from \cite[Theorem 2]{gubler_soto13} that every point of $\Ycal$ admits an open affine $\tdop$-invariant neighborhood corresponging to some $\Gamma$-admissible cone. Finally in \cite[Theorem 3 ]{gubler_soto13} it is shown that all these cones form a $\Gamma$-admissible fan $\Sigma$ proving the statement.
\end{proof}

\begin{remark}
The extra condition on the cones, stated in the theorem, is required in order to guarantee that the toric scheme constructed from the $\Gamma$-admissible fan $\Sigma$ is of finite type over $\kcirc$.
\end{remark}

\art \label{properness} It is well known that a normal toric variety over a field is proper if and only if the associated fan is complete, i.e. it has support $N_\rdop$. For $\tdop$-toric varieties over $\kcirc$, properness is characterized in a similar way. We say that a $\Gamma$-admissible fan $\Sigma$ is complete if its support is $N_\rdop \times \rdop_+$. A $\tdop$-toric variety $\Ycal_\Sigma$ is universally closed over $\kcirc$ if and only if $\Sigma$ is a complete $\Gamma$-admissible fan, see \cite[Proposition 11.8]{gubler12}. In this case the generic and the special fiber of $\Ycal_\Sigma$ are proper schemes over $K$ and $\ktilde$ respectively. If the valued group is discrete or divisible, completeness of the fan is equivalent to being proper over $\kcirc$. If the $\Gamma$-admissible fan $\Sigma$ is complete and consist of cones satisfying the extra condition stated in the previous theorem, then the $\tdop$-toric variety $\Ycal_\Sigma$ is proper over $\kcirc$.

\art \label{projective} Projective toric varieties over a field can be obtained by taking closures of torus orbits in projective space. In general we may end up with a non-normal toric variety, which also admits a very neat combinatorial description. Actually, from this description it is possible to obtain in a canonical way the normalization of the given variety, see \cite[Proposition 4.9]{gkz}. 
We briefly review the construction over $\kcirc$, for details see \cite[\S 9]{gubler12}. 

Let $A=(m_0,\ldots, m_N)\in M^{N+1}$ and consider the action of $T$ on a point $\y=(y_0:\cdots :y_N)\in \pdop^N_\kcirc(K)$ given by
\[ t\cdot \y:=(\chi^{m_0}(t)y_0:\cdots :\chi^{m_{N}}(t)y_N).\]
By taking the closure of $T\cdot \y$ in $\pdop^N_\kcirc$ we get a projective $\tdop$-toric variety over $\kcirc$ with a linear action of $\tdop$. It does not depend on the point $\y \in \pdop^N_\kcirc(K)$ but on the valuation of its coordinates. This information is encoded on the height function defined by $a:\{0,\ldots , N\} \to \Gamma\cup \{\infty\}$, $j\mapsto v(y_j)$. This projective toric variety is denoted by $\Ycal_{A,a}$. Every projective $\tdop$-toric variety over $\kcirc$ with a linear action of the torus is of this form, see \cite[Proposition 9.8]{gubler12}. 

The combinatorial description of a projective $\tdop$-toric variety with a linear action of the torus is given as follows. First let us consider the weight polytope in $M_\rdop$ given by 
\[ \te{Wt}(\y):=\te{Conv}(A(\y)),\]
where $A(\y):=\{m_i\in M|y_i\neq0\}$. With the height function $a$ we subdivide this polytope by projecting the faces of the convex hull of $\{(m_i,\lambda_i)\in M_\rdop \times \rdop_+|\lambda_i \geq a(i)\}$ into $M_\rdop$. This subdivided weight polytope is denoted by $\te{Wt}(\y,a)$. Dually, we get a polyhedral complex in $N_\rdop$ as the domain of linearity of the piecewise linear function $g:N_\rdop \to \rdop$, given by $w\mapsto g(w):=\te{min}\{a(i)+\langle m_i, w\rangle \}$. We denote this polyhedral complex as $\Ccal(A,a)$. It is dual to $\te{Wt}(\y,a)$ in the sense that cones of $\Ccal(A,a)$ of dimension $d$ correspond to faces of $\te{Wt}(\y,a)$ of dimension $n-d$. Explicitely, given a face $Q$ of $\te{Wt}(\y,a)$ we have the cone $\sigma_Q$ defined by 
\[ \{ w\in N_\rdop | g(w)=\langle m_i, w\rangle +a(i), i\in Q\cap A(\y)\}.\]

Dually, given a cone $\sigma \in  \Ccal(A,a)$ we get the face $Q_\sigma$ given by the convex hull of
\[ \{m_i\in A(\y)|g(w)=a(i)+\langle m_i, w \rangle , \forall w\in \sigma \}. \]

Now, we can describe the torus orbits of $\Ycal_{A,a}$ as follows. The $T$-orbits of the generic fiber $({\Ycal_{A,a}})_\eta$ are in one-to-one correspondence with the faces of the weight polytope $\te{Wt}(\y)$, hence with the cones of its normal fan. The $\tdop$-orbits of the special fiber $({\Ycal_{A,a}})_s$ are in one-to-one correspondence with the faces of the weight subdivided polytope $\te{Wt}(\y,a)$, hence with the polyhedra of the polyhedral complex $\Ccal(A,a)$.  Note that the irreducible components of the special fiber are in one-to-one correspondence with the maximal cells of the subdivided weight polytope. Let $\Sigma(A,a)$ be  the fan generated by $\Ccal(A,a)$ in $N_\rdop \times \rdop_+$. It follows from the construction that it is a $\Gamma$-admissible fan. We've seen that the $\tdop$-toric variety associated to a $\Gamma$-admissible fan is normal. Thus $\Ycal_{\Sigma(A,a)}$ is normal.

In general, the $\tdop$-toric variety $\Ycal_{A,a}$ is not normal. The relation between the $\tdop$-toric varieties $\Ycal_{\Sigma(A,a)}$ and $\Ycal_{A,a}$ is given in the following proposition. This generalizes  a result of Qu in the case of discrete valuations, see \cite[\S 2.3]{qu}.

\begin{prop}
\label{normalization}
Let $\Ycal_{A,a}\hookrightarrow \pdop^N_\kcirc$ be the projective $\tdop$-toric variety over $\kcirc$ associated to $A\in M^{N+1}$ and to a height function $a$. Then the normal $\tdop$-toric variety $\Ycal_{\Sigma(A,a)}$ associated to the $\Gamma$-admissible fan $\Sigma(A,a)$ is the normalization of $\Ycal_{A,a}$ over $\kcirc$.
\end{prop}
\begin{proof}
Let $\{ u_i\}_{i=0}^N$ be the set of vertices of $\te{Wt}({\bf{y}},a)$ and let $z_0,\ldots, z_N$ be the coordinates of $\pdop^N_\kcirc$. Let $y=(y_0:\cdots :y_N)\in \mathbb{P}^N_\kcirc(K)$ be a point associated to the height function $a$, i.e. $v(y_j)=a(j)$ for $j=0, \ldots , N$. Consider $\Ucal_i\simeq \mathbb{A}^N_\kcirc$ to be the affine open subscheme in $\pdop^N_\kcirc$ given by $\{z_i\neq 0\}$. By denoting $x_k:=z_k/z_i$, we have $\Ucal_i=\te{Spec}(\kcirc[x_0,\ldots, \widehat{x_i},\ldots, x_N])$ and $\tdop=\Spec(\kcirc [M])=\Spec(\kcirc[{x}^{\pm  1}_0,\ldots, \widehat{x_i},\ldots, x_N^{\pm  1}])$. The $\tdop$-toric variety  $\Ucal_i \cap \Ycal_{A,a}$ is isomorphic to the closure of $T\cdot y^{(i)}$ in $\mathbb{A}^N_\kcirc$, with \[y^{(i)}=(y_0/y_i,\ldots, \widehat{y_i/y_i},\ldots, y_N/y_i) \in \adop^N_\kcirc (K)\] under the action \[t\cdot y^{(i)}:=(\chi^{m_0-m_i}(t)y_0/y_i,\ldots , \chi^{m_N-m_i}(t)y_N/y_i).\] This action is equivalent to the morphism
\[ K[x_0,\dots, \widehat{x_i},\ldots,x_N]\to K[M],\quad x_j\mapsto y_j/y_i\chi^{m_i-m_j}. \]
The closure of this orbit in $\mathbb{A}^N_\kcirc$ is given by the image of the induced map
\[ \kcirc[x_1,\dots,\widehat{x_i},\ldots, x_N]\to K[M],\]
which is $B_i:=\kcirc[y_1/y_i\chi^{m_1-m_i},\ldots ,y_N/y_i\chi^{m_N-m_i}]$. 

Now, recall that given a semigroup $S\subset M\times \Gamma$, the $\kcirc$-algebra $\kcirc[S]$ is normal over $\kcirc$ if and only if $S$ is saturated: it is proved in \cite[Lemma 4.1]{gubler_soto13} that normality implies saturation, the converse follows the same lines as in the proof of \cite[Lemma 1]{mumford73}. From the proof of \cite[Proposition 4.4]{gubler_soto13} we know that for any semigroup $S$, its saturation is given by $\te{cone}(S)\cap (M\times \Gamma)$. In our case $S$ is the semigroup given by $\{(m,v(b))|b\chi^m \in B_i\}$. By \cite[Lemma 4.2]{gubler_soto13}, the cone generated by $S$ is equal to the cone generated by 

\[\{(0,1),(m_j-m_i,v(y_j)-v(y_i))|j\in\{1,\ldots,\widehat{i},\ldots n\} \}\subset M\times \Gamma\]
 in $M_\rdop \times \rdop$. We have that $K^\circ[\te{cone}(S)\cap (M\times \Gamma)]$ is the same as $K[M]^{\sigma_i}$, where $\sigma_i$ is the dual cone of $\te{cone}(S)$. As this cone is $\Gamma$-admissible, it follows that $\Ucal_{\sigma_i}$ is an affine normal $\tdop$-toric variety over $\kcirc$. Let us see that this is in fact the normalization of $\Ucal_i \cap \Ycal_{A,a}$. First note that $K[M]^{\sigma_i}$ is integral over $B_i$. In fact, any element $f\in K[M]^{\sigma_i}$ can be written as

\[ f=c(y_1/y_i\chi^{m_1-m_i})^{\lambda_1}\cdots (y_1/y_i\chi^{m_N-m_i})^{\lambda_N},\]
where $\lambda_i \in \mathbb{Q}_+$ for $i=1,\ldots, N$ and $c\in K$; see \cite[Proposition 4.4]{gubler_soto13}. Then $f^k\in B_i$ for $k$ large enough. This means that $f$ is a root  of the polynomial  $T^k-f^k\in B_i[T]$. 
Now take an element $f\in \te{Frac}(B_i)=K(M)$ which is integral over $B_i$. Without loss of generality we may assume that it is of the form $b\chi^m$, for some $m\in M$ and $b\in K$. It satisfies 

\[ (b\chi^m)^n+f_{1}(b\chi^m)^{n-1}+\cdots +f_{n-1} (b\chi^m) +f_{n}=0, \quad \te{with $f_1,\ldots, f_{n}\in B_i$.}\]
It follows that $r(m, v(b))\in \te{cone}(S)\cap (M\times \Gamma)$, for some $r\in \zdop_+$. By saturation, we have that $(m,v(b))\in \te{cone}(S)\cap (M\times \Gamma)$, then $b\chi^m\in K[M]^{\sigma_i}$.

Doing this on each coordinate chart in $\mathbb{P}^N_{K^\circ}$, we get a normal $\tdop$-toric variety $\Ycal_\Sigma$ over $\kcirc$ associated to the fan $\Sigma(A,a)=\{\sigma_i| u_i \te{ vertex of } \te{Wt}({\bf{y}},a)\}$. To end the proof, we just need to show that $\Sigma_1(A,a)=\Ccal(A,a)$. But this is clear from the construction. Explicitely, we have

\begin{flalign*}
 \sigma_i &=\left\{ (\omega , t)\in N_\rdop \times \rdop | \langle m_j-m_i,\omega \rangle +t(a(j)-a(i))\geq 0, \forall j \in\{1,\ldots,\widehat{i},\ldots n\} \right\}\\ 
 &= \left\{ (\omega , t)\in N_\rdop \times \rdop |\langle m_j,\omega \rangle + ta(j)\geq \langle m_i,\omega \rangle + ta(i), \forall j \in\{1,\ldots,\widehat{i},\ldots n\}\right\}
\end{flalign*}
then, the polyhedron at level 1 is given by
\begin{flalign*}
 (\sigma_i)_1&=\left\{ \omega \in N_\rdop | \langle m_j,\omega \rangle +a(j)\geq \langle m_i,\omega \rangle + a(i), \forall j \in\{1,\ldots,\widehat{i},\ldots n\} \right\}\\ 
 &= \left\{ \omega \in N_\rdop |g(\omega)=  \langle m_i,\omega \rangle + a(i) \forall j \in\{1,\ldots,\widehat{i},\ldots n\}\right\}
\end{flalign*}
which is the corresponding polyhedron of the complex $\Ccal(A,a)$.
\end{proof}

\section{$\tdop$-invariant blow-ups and Zariski--Riemann spaces}

In this section we will review some classical results on blow ups of toric varieties along invariant closed subschemes adapted to the setting of $\tdop$-toric varieties over $\kcirc$. They are fundamental elements in the proof of Nagata's compactification theorem as can be seen in \cite{conrad07} and \cite{sumihiro2}. After proving some basic properties of admissible blow ups, we will be able to define the $\tdop$-invariant Zariski--Riemann space associated to a pair $(\Ycal, \Ucal)$, where $\Ycal$ is a $\tdop$-toric variety over $\kcirc$ and $\Ucal \subset \Ycal$ is a $\tdop$-invariant open dense subset. This locally ringed space,  canonically endowed with a $\tdop$-action, is our main tool for proving Theorem 1.

We recall that the blow up of a scheme $\Ycal$ over $\kcirc$ with center on a closed subscheme $\Zcal$ with coherent ideal sheaf $\mathscr{I}$ is given by the projective morphism $\te{Bl}_\Zcal (\Ycal):= \te{Proj}(\bigoplus \mathscr{I}^n ) \xrightarrow{\beta} \Ycal$. Note that it is of finite type as the ideal sheaf is coherent. 

\begin{prop}\label{blow-up is toric}
Let $\Ycal$ be a $\tdop$-toric variety over $K^\circ$ and let $\Zcal \subset \Ycal$ be a $\tdop$-invariant closed subscheme of finite type over $K^\circ$ with ideal sheaf $\mathscr{I}$. Then the blow up of $\Ycal$ with center $\Zcal$ is a $\tdop$-toric variety over $K^\circ$. 
\end{prop}
\begin{proof}
We have $\te{Bl}_{\Zcal}(\Ycal)=\te{Proj}(\bigoplus \mathscr{I}^n)$ over $\Ycal$.  Since $\Zcal$ is $\tdop$-invariant, its ideal sheaf $\mathscr{I}$ has an $M$-graduation. Then, clearly so does $\bigoplus \mathscr{I}^n$ and hence $\te{Bl}_{\Zcal}(\Ycal)$ has a $\tdop$-action over $\kcirc$ extending the multiplication action of $T$ on itself. It is flat over $\kcirc$ as it is torsion free and the finite type property follows from the fact that $\Zcal$ is of finite type over $K^ \circ$.
\end{proof}

\begin{remark}
It is well known that, in general, the blow up of a normal toric variety along a closed invariant center is not necessarily normal, e.g. the blow up of  $\mb{A}^2_K$ along $(x^2, y^2)$ is not normal. As we don't require normality in the definition of $\tdop$-toric varieties, the previous proposition shows that  we remain in the category of $\tdop$-toric varieties after performing a blow up along an invariant center.  
\end{remark}

\art Note that given a $\tdop$-invariant open subset $\Ucal$ in a $\tdop$-toric variety $\Ycal$ over $\kcirc$, there is a coherent $\tdop$-invariant ideal sheaf $\Ical$ such that $\te{V}(\Ical)=\Ycal	 \backslash \Ucal$. It is clear that the ideal sheaf $\Ical$ associated to the closed subset $\Ycal \backslash \Ucal$ is $\tdop$-invariant, we just need to see that it is of finite type. Since the problem is local, we may assume that $\Ycal=\te{Spec}(A)$ and that $\Ical$ is the ideal sheaf associated to an ideal $I\subset A$. In general we have $I=\lim\limits_{\longrightarrow} I_k$, for finitely generated ideals $I_k$. By the same arguments as in \cite[Lemma 1.3]{conrad07} we see that there is one element $I_k$ big enough in this family, such that $\te{V}(I_k)=\te{V}(I)$. It is clear that $I_k$ is a $\tdop$-invariant ideal. Note that the ideal $I_k$ is not canonical, in particular there may be many of them satisfying this condition.

\art \label{strict transform} Consider a blow up $\beta:\Ycal' \to \Ycal$, with center $\Zcal$ of finite type over $\kcirc$.  If $\Xcal \hra \Ycal$ is a closed subscheme, the strict transform is defined as the scheme-theoretic closure of $ \beta^{-1}(\Xcal \backslash \Zcal)$ in $\Ycal'$ and is denoted by $\Xcal'$. Equivalently, it is the blow up of $\Xcal$ with center $\Xcal \cap \Zcal$. In order to make an explicit description of the strict transform, let us consider an affine open $\te{Spec}(A)$ of $\Ycal$ and let $I=\scr{I}(\te{Spec}(A))\subset A$. Let $f_1,\ldots, f_k$ be a set of generators of $I$. It is well known that the blow up over $\te{Spec}(A)$ is covered by the open schemes  $\te{Spec}(A[f_1/f_i,\ldots ,f_k/f_i])$, for $i=1, \ldots, k$.  If on $\te{Spec}(A)$ the  closed subscheme $\Xcal$ is defined by the ideal $J$, then the strict transform $\Xcal'$ over $\te{Spec}(A)$ is cut out locally by the ideals $J_{f_i}\cap A[f_1/f_i,\ldots ,f_k/f_i]$ in $\te{Spec}(A[f_1/f_i,\ldots ,f_k/f_i])$, where $J_{f_i}$ is the image of the ideal $J$ in $A_{f_i}$ under the canonical map $A\to A_{f_i}$. From this local description, we see that if the blow up $\beta:\Ycal' \to \Ycal$ and the closed subscheme $\Xcal$  are compatible with the $\tdop$-action, then the strict transform  and the induced morphism $\Xcal' \to \Ycal'$ are $\tdop$-equivariant.

 In toric geometry, one is interested in the study of blow ups of toric varieties along closed invariant centers.  The complement of the center can be identified with an open invariant subscheme of the blow up. We will focus on blow ups which do not modify a given $\tdop$-invariant open subscheme. 
 
\begin{definition}\label{admissible}
 Let $\Ucal \subset \Ycal$ be a $\tdop$-invariant. 
 A blow up of $\Ycal$ along a closed  $\tdop$-invariant center contained in $\Ycal \backslash \Ucal$ is called \emph{$\Ucal$-admissible}.   
\end{definition}

\begin{remark}
This definition differs slightly from the standard notion of $\Ucal$-admissible blow up, where the center is not required to be $\tdop$-invariant. Note that the center of the blow up can be strictly contained in $\Ycal \backslash \Ucal$, therefore not every center disjoint from $\Ucal$ is necessarily $\tdop$-invariant. 
\end{remark}

In the remainder of this section, we fix a $\tdop$-toric variety $\Ycal$ over $\kcirc$ and $\Ucal \subset \Ycal$ a $\tdop$-invariant open subscheme. In what follows, we will prove some basic properties of $\Ucal$-admissible blow ups. The next proposition shows that blow ups of $\tdop$-invariant closed subschemes of a $\tdop$-toric variety can be extended to blow ups of the whole variety in a compatible way with the torus action.  More precisely, we have the following result.

\begin{prop}\label{blow up extension}
 Let $\Xcal \hra \Ycal$ be a $\tdop$-equivariant closed immersion and let $\Xcal' \to \Xcal$ be a $\Vcal$-admissible blow up, with $\Vcal=\Ucal\cap \Xcal$. Then there is a $\Ucal$-admissible blow up $\Ycal' \to \Ycal$ such that the following diagram 

 \[ \xymatrix{\Xcal' \ar[d]\ar@{^{(}->}[r]^{j'} & \Ycal' \ar[d] \\ \Xcal \ar@{^{(}->}[r]^j& \Ycal} \]
 commutes and is compatible with the torus action.
\end{prop}

\begin{proof}
 From \cite[Proposition E.1.6]{fujiwara-kato14} it follows that there exists a blow up $\Ycal'\to \Ycal$ with center disjoint from $\Ucal$ such that it extends $\Xcal' \to \Xcal$ making the above diagram commutative. With our hypothesis, the same construction gives rise to a diagram that is actually $\tdop$-equivariant.   Explicitly, the construction goes as follows: let $\Jcal$ be the ideal sheaf of the  center of the $\Vcal$-admissible blow up $\Xcal' \to \Xcal$. Let $\Ical$ be a $\tdop$-invariant quasi-coherent ideal sheaf on $\Ycal$  with support $\Ycal \backslash \Ucal$ such that $\Ical_{|{\Xcal}}\subset \Jcal$.  The quotient ideal sheaf  $\Jcal/\Ical_{|\Vcal}$ in $\te{V}(\Ical)\cap \Xcal$, which is $\tdop$-invariant, can be extended to a $\tdop$-invariant sheaf $\widetilde{\Jcal / \Ical_{|\Vcal}}$ on $\te{V}(\Ical)$.  This follows from [EGA1New 6.9.2] and by a limit argument we can assume that it is of finite type. 
   
As the closed embedding $\te{V}(\Ical)\ra \Ycal$ is $\tdop$-equivariant, the push forward $\Kcal$ of $\widetilde{\Jcal / \Ical_{|_\Xcal}}$ is a $\tdop$-invariant ideal sheaf on $\Ycal$.  It is clear that the blow up with center $\Kcal$ is $\Ucal$-admissible and makes the diagram $\tdop$-equivariant.
\end{proof}

As we may expect, $\Ucal$-admissible blow ups are stable under composition. This is the content of  the following proposition.
\begin{prop}
Let $f: \Ycal' \to \Ycal$ be a $\Ucal$-admissible blow up and $g: \Ycal'' \to \Ycal'$ be a $f^{-1}(\Ucal)$-admissible blow up, then $g\circ f:\Ycal'' \to \Ycal$ is a $\Ucal$-admissible blow up.
\end{prop}
\begin{proof}
Let $\Ical, \Ical'$ be the centre of the blow ups $f$ and $g$ respectively. From \cite[Premi\`ere partie, Lemme (5.1.4)]{raynaud} there is a quasi-coherent ideal sheaf $\Ical''\subset \Ocal_\Ycal$ such that $\Ical'' \Ocal_{\Ycal'}=\Ical^n {\Ical'}^m\Ocal_{\Ycal'}$ for some $n, m>0$. Furthermore,  $g\circ f$ is a blow up with centre $\Ical \Ical''$. As $\Ical$ and $\Ical'$ are coherent, then so is $\Ical''$. It is also clear that if $f$ and $g$ are blow ups with $\tdop$-invariant centres, then so is $g\circ f.$
\end{proof}

Given a morphism which is a local open immersion, we can use blow ups to extend it to an open immersion. 

\begin{lemma}\label{extension:immersion}
 Let $\Ycal$ be a $\tdop$-toric variety and let $\Ucal\subset \Ycal$ be an open $\tdop$-invariant subscheme. Consider a $\tdop$-equivariant morphism $f:\Zcal \to \Ycal$ of $\tdop$-toric varieties of finite type such that the induced morphism $f^{-1}(\Ucal)=\Ucal \times_{\Ycal} \Zcal \to \Ucal$ is an open immersion. Then there exists a $\Ucal$-admissible blow up $ \Ycal' \to \Ycal$ such that the induced morphism $\Zcal' \to \Ycal'$ is an open immersion, where $\Zcal'$ is the strict transform of $\Zcal$. 
\end{lemma}

\begin{proof}

The existence of the $\Ucal$-admissible blow up giving rise to the open immersion follows from \cite[Premi\`ere partie, Corollaire 5.7.11]{raynaud}. Note that it is admissible in our sense, as everything is compatible with the action of the torus. 
From the construction of the $\Ucal$-admissible blow up, we have that the induced morphism is $\tdop$-equivariant.
\end{proof}

Now, given two $\Ucal$-admissible blow ups, there is a $\Ucal$-admissible blow up which dominates both. This shows that the collection of $\Ucal$-admissible blow ups is filtered.

\begin{prop}\label{blow-up domination}
 Let $\beta_1:\Ycal_1\to \Ycal$  and $\beta_2:\Ycal_2 \to \Ycal$ be two $\Ucal$-admissible admissible blow ups along the ideal sheaves $\Ical_1$ and $\Ical_2$ respectively. Then there is a $\Ucal$-admissible blow up $\beta:\Ycal_{12}\to \Ycal$ such that the following diagram commutes
 \[ \xymatrix{\Ycal_{12} \ar[d]^{\beta_2'} \ar[r]^{\beta_1'}&\Ycal_2 \ar[d]^{\beta_2} \\ \Ycal_1 \ar[r]^{\beta_1}&\Ycal }\]
\end{prop}

\begin{proof}
 It follows from the proof of \cite[Proposition E.2.1]{fujiwara-kato14} that $\beta_2:\Ycal_{12}= \te{Bl}_{\Ical_1 \Ical_2}(\Ycal)\to \Ycal$ induces a commutative diagram as above. Since $\beta_1, \beta_2$ are $\Ucal$-admissible blow ups, it follows that $\beta_{12}$ is $\Ucal$-admissible as well.  Clearly, the induced morphisms are $\tdop$-equivariant.
\end{proof}

A $\tdop$-invariant coherent ideal sheaf $\Ical$ with $\te{V}(\Ical)\cap \Ucal =\emptyset$ induces a $\Ucal$-admissible blow up. Let $\tdop$-$\te{\bf AId}_{(\Ycal, \Ucal)}$ be the set of coherent ideal sheaves inducing $\Ucal$-admissible blow ups. 

We endow this set with an order relation as follows: we say that $\Ical \leq \Ical'$ if there is a $\Ucal$-admissible ideal sheaf $\Ical''$ such that $\Ical =\Ical'\Ical''$. That is, if the $\Ucal$-admissible blow up with centre $\Ical$ can be obtained from $\Ical'$ by modifying its centre in a precise way. Given an ideal sheaf $\Ical_1$ smaller than or equal to an ideal sheaf $\Ical_2$, it follows from Proposition \ref{blow-up domination} that we have an induced morphism $\beta_{21}: \te{Bl}_{\Ical_1}(\Ycal)\to \te{Bl}_{\Ical_2}(\Ycal)$. Since those blow ups over $\Ycal$ are $\Ucal$-admissible, we can see that the corresponding morphism $\beta_{21}$ is $\Ucal$-admissible as well by identifying $\beta^{-1}_i(\Ucal)$ with $\Ucal$, where $\beta_i:\te{Bl}_{\Ical_i}(\Ycal)\to \Ycal$, for $i=1,2$. 

We have the inverse system $\{ \Ycal_i, \beta_{ij}\}$, where $\beta_{ij}: \Ycal_i \to \Ycal_j$ are the morphisms described above. They are compatible with the corresponding $\Ucal$-admissible blow ups $\beta_i:\Ycal_i\to \Ycal$. 

\begin{definition}\label{def:zariski-riemann}
 
The \emph{$\tdop$-invariant Zariski--Riemann} space associated to the pair $(\Ycal, \Ucal)$ is given by the limit

\[ \langle \Ycal \rangle_{\Ucal}:=\varprojlim \Ycal_i,\] 
in the category of locally ringed spaces. Using the canonical projections $\pi_i : \langle \Ycal \rangle_{\Ucal} \to \Ycal_i$, the structure sheaf $\Ocal_{\zr{\Ycal}{\Ucal}}$ is $\Ocal_{\langle \Ycal \rangle_{\Ucal}}:=\varinjlim \pi_i^{-1}\Ocal_{\Ycal_i}$. Note that the stalks of this sheaf are given as the direct limits of the stalks on each space $\Ycal_i$. That is, for a point $y\in \zr{\Ycal}{\Ucal}$ we have  $\Ocal_{\zr{\Ycal}{\Ucal} ,y}=\varinjlim \Ocal_{{\Ycal_i},{y_i}}$, with $y_i=\pi_i(y)$.

\end{definition}

\begin{remark}
The inverse system $\{\Ycal_i, \beta_{ij}\}$ is compatible with the action of the torus $\tdop$ as the induced morphisms are $\tdop$-equivariant. In addition, since $\te{Hom}$ commutes with inverse limits, the space $\langle \Ycal \rangle_{\Ucal}$ is a locally ringed space endowed with a  canonical $\tdop$-action given by  $(t,(x_i))\mapsto (tx_i)$. 
\end{remark}

One of the main features of the $\tdop$-invariant Zariski--Riemann space is the following well known fact.

\begin{prop}
Keeping the notation of the definition \ref{def:zariski-riemann}, the $\tdop$-invariant Zariski--Riemann space $\zr{\Ycal}{\Ucal}$ is quasi-compact.
\end{prop}

\begin{proof}
It follows from \cite[0.2.2.10]{fujiwara-kato14} by noting that every $\tdop$-toric scheme $\Ycal_i$ is a coherent and sober space. 
\end{proof}

The following notion will be important for the gluing of $\tdop$-invariant Zariski--Riemann spaces.
\begin{definition}
Let $\Ucal \hra \Ycal$ be a quasi-compact open immersion and let $\Ucal \to \Xcal$ be a morphism. Consider the induced map
\[ i:\Ucal \to \Ycal \times_{\kcirc} \Xcal\]
which is an immersion as well. The closure of $\Ucal$ in $\Ycal \times_\kcirc \Xcal$ is called \emph{the join of $\Ycal$ and $\Xcal$  along $\Ucal$} and is denoted by $\Ycal *_{\Ucal} \Xcal$. 
If $\Ycal$, $\Xcal$ are $\tdop$-toric varieties over $\kcirc$, $\Ucal$ is $\tdop$-invariant and the morphisms $\Ucal \hra \Ycal$ and $\Ucal \to \Xcal$ are $\tdop$-equivariant, then the induced map $i:\Ucal \to \Ycal \times_{\kcirc} \Xcal$ is $\tdop$-equivariant. In this case the join $\Ycal *_\Ucal \Xcal$ is a closed $\tdop$-invariant subscheme of $\Ycal \times_\kcirc \Xcal$.
\end{definition}

We have the following $\tdop$-equivariant diagram 

\begin{equation}\label{diagram blow-ups}
\begin{gathered}
 \xymatrix{ & \Ycal & \\\Ucal \ar@{^{(}->}[ur] \ar[dr] \ar^i[rr]&  & \Ycal *_\Ucal \Xcal\subset \Ycal \times_\kcirc \Xcal \ar^p[ul] \ar^q[dl] \\& \Xcal&} 
\end{gathered}
 \end{equation}
where $p$ and $q$ are the canonical projections.  We will see that after performing a $\Ucal$-admissible blow up, the morphisms in the upper triangle of the diagram become open immersions compatible with the $\tdop$-action. More precisely we have the following proposition.

\begin{prop}\label{diagram immersions}
 Consider the diagram (\ref{diagram blow-ups}) above. There exists a $\Ucal$-admissible blow up $\Ycal' \to \Ycal$ such that the strict transform $p'\colon \Zcal \to \Ycal'$ of  $\Ycal*_{\Ucal} \Xcal \xrightarrow{p}  \Ycal$ and the induced morphism $\Ucal \to \Zcal$ are $\tdop$-equivariant open immersions. If $\Xcal$ is proper over $\kcirc$, then the induced morphism $p'$ is an isomorphism.
 
If in addition the morphism $\Ucal \to \Xcal$ is an open immersion, there is a $\Ucal$-admissible blow up $\Xcal' \to \Xcal$ such that the strict transform $\Zcal' \to  \Xcal'$ of  $ \Zcal \to  \Xcal$ is an open immersion as well. Furthermore, we may assume that $\Zcal' \to \Ycal'$ is a $\tdop$-equivariant open immersion. 
\end{prop}

\begin{proof}
 By the Lemma \ref{extension:immersion}, there is a $\Ucal$-admissible blow up $\Ycal' \to \Ycal$ such that the strict transform $p'\colon \Zcal \to \Ycal'$ is a $\tdop$-equivariant  open immersion. If $\Ycal$ is  proper over $\kcirc$, then the induced morphism $p'$ is proper, therefore it is an isomorphism. Now suppose that $\Ucal \to \Xcal$ is an open immersion, by the same argument there is a $\Ucal$-admissible blow up $\Xcal' \to \Xcal$ such that the strict transform $\Zcal'\to \Xcal'$ of $\Zcal \to \Xcal$ is a $\tdop$-equivariant open immersion. It follows from Proposition \ref{blow up extension}, that the $\Ucal$-admissible blow up $\Zcal' \to \Zcal$ can be extended to a $\Ucal$-admissible blow up $\Ycal'' \to \Ycal'$ such that $\Zcal' \to \Ycal''$ is an open immersion compatible with the $\tdop$-action. As $\Ucal$-admissible blow ups are stable under composition, we may assume that $\Zcal' \hra \Ycal'$. 
\end{proof}

Suppose that $\Ucal \subset \Ycal$ is compactifiable  by a proper $\tdop$-toric scheme $\overline{\Ucal}$ over $\kcirc$ and set 
\[ \langle \Ucal \rangle_{\te{cpt}} := \zr{\overline{\Ucal}}{\Ucal}. \]
It is called the \emph{canonical compactification} of $\Ucal$ over $\kcirc$. It follows from Proposition \ref{diagram immersions} that this locally ringed space is independent of the choice of an algebraic compactification of $\Ucal$.

\begin{lemma}\label{embedcan}
With the notation above, we have $\zr{\Ycal}{\Ucal} \hra \cpt{\Ucal}$ $\tdop$-equivariantly.
\end{lemma}
\begin{proof}
By Proposition \ref{diagram immersions}, after $\Ucal$-admissible blow ups we have  an open embedding $\Ycal' \hra {\overline{\Ucal}}'$, where $\Ycal' \to \Ycal$ and ${\overline{\Ucal}}' \to {\overline{\Ucal}}$ are $\Ucal$-admissible blow ups. By considering the inverse systems giving rise to $\cpt{\Ucal}$ and $\zr{\Ycal}{\Ucal}$, it is clear that we have an embedding $\zr{\Ycal}{\Ucal} \hra \cpt{\Ucal}$. As all the maps are $\tdop$-equiraviant, so is the induced map.
\end{proof}

\begin{prop}
Let $\Ucal \hra \Ycal_1$ and  $\Ucal \hra \Ycal_2$ be $\tdop$-equivariant open immersions into $\tdop$-toric varieties over $\kcirc$, then we have
\[ \langle \Ycal_1 *_\Ucal \Ycal_2 \rangle_\Ucal= \langle \Ycal_1  \rangle_\Ucal \cap  \langle  \Ycal_2 \rangle_\Ucal\]
in $\cpt{\Ucal}$.
\end{prop}

\begin{proof}
By the Proposition \ref{diagram immersions}, there exist  $\Ucal$-admissible blow ups $\Ycal_i' \to \Ycal_i$, $i=1,2$,  such that the induced morphisms $\Zcal \to\Ycal_i'$, $i=1,2$ are open immersions, with $\Zcal \to  \Ycal_1 *_\Ucal \Ycal_2 $ the strict transform. Then we have induced maps $\zr{\Ycal_1*_\Ucal \Ycal_2}{\Ucal} \to \zr{\Ycal_i}{\Ucal}$, $i=1,2$ which are open immersions.  This implies that $ \langle \Ycal_1 * \Ycal_2 \rangle_\Ucal \subset \langle \Ycal_1  \rangle_\Ucal \cap  \langle  \Ycal_2 \rangle_\Ucal $. The other inclusion is clear as we have $\Ucal \subset \Ycal_1 *_\Ucal \Ycal_2$.
\end{proof}

One of the most important propositions for constructing the $\tdop$-equivariant completion of $\Ycal$ is the following. 
\begin{prop}\label{gluing}

Keeping the notation from the last proposition, there exists a $\tdop$-equivariant open immersion $\Ucal \hra \Zcal$ such that
\[ \zr{\Zcal}{\Ucal}=\zr{\Ycal_1}{\Ucal} \cup \zr{\Ycal_2}{\Ucal}\]
\end{prop}

\begin{proof}
From \cite[Lemma F.3.2]{fujiwara-kato14} it follows that an open embedding with the required property exists. Explicitly, after $\Ucal$-admissible blow ups  $\Ycal_i' \to \Ycal_i	$, the strict transform $\Zcal' \to \Ycal_1 *_\Ucal \Ycal_2 $ admits open immersions into $\Ycal_i'$, for $i=1,2$. Then $\Zcal$ is given by the gluing of $\Ycal_1'$ and $\Ycal_2'$ along $\Zcal'$. Since $\Ycal_1'$, $\Ycal_2'$ and $\Zcal'$ are $\tdop$-toric varieties, after gluing we end up with a $\tdop$-toric variety as well. It is clear that all the morphisms are compatible with the $\tdop$-action.
\end{proof}
This proposition allow us to glue schemes keeping track of the  $\tdop$-invariant Zariski--Riemann spaces. This will be crucial for the construction of the proper $\tdop$-toric variety containing $\Ycal$ as an open and dense subscheme.

\section{Equivariant compactification}

In this section, we will use the results on $\tdop$-invariant Zariski--Riemann spaces in order to prove that every normal $\tdop$-toric variety can be embedded into a proper $\tdop$-toric variety over $\kcirc$. Here we follow the lines of the proof of Nagata's embedding theorem given by Fujiwara--Kato in \cite[Appendix F]{fujiwara-kato14}, adapted to our setting.  The main idea is to construct first a locally ringed space $\cpt{\Ycal}$ containing $\Ycal$ as an open dense subset, see Definition \ref{cp} below. After that, we construct a $\tdop$-toric variety ${\Ycal}_\te{cpt}$ proper over $\kcirc$ such that $\zr{{\Ycal}_\te{cpt}}{\Ycal}=\cpt{\Ycal}$ as locally ringed spaces.

\art We fix a normal $\tdop$-toric variety $\Ycal$ over $\kcirc$. Let $\Ucal \subset \Ycal$ be a $\tdop$-invariant open affine subset. By \cite[Theorem 1]{gubler_soto13} we know that it is an affine normal $\tdop$-toric variety. We can take its closure $\overline{\Ucal}$ in some projective space over $\kcirc$. 
From \cite[Theorem 2]{gubler_soto13}, we know that  $\Ycal$ admits an affine $\tdop$-invariant open covering, which implies that any $\Ucal$-admissible blow up $\beta: \Ycal' \to \Ycal$  also does. This follows from the fact that $\beta^ {-1}(\Ucal_i)\to \Ucal_i$ is the blow up of $\Ucal_i$ along $\Ical_{|\Ucal_i}$, with $\Ical$ the center of $\beta$, which is $\tdop$-invariant.

Note that from Lemma \ref{embedcan}, it follows that for an affine $\tdop$-invariant open covering $\{\Ucal_i\}$ of $\Ycal$, we have canonical open immersions $\zr{\Ycal}{\Ucal_i} \hra \cpt{\Ucal_i}$.

\begin{definition}\label{cp}

The \emph{partial compactification} of $\Ucal$ relative to $\Ycal$ is
\[ \pc{\Ucal}{\Ycal}:=\cpt{\Ucal} \backslash \overline{\zr{\Ycal}{\Ucal}\backslash \Ucal}.\]
\end{definition}

\begin{remark}
The locally ringed spaces $\langle \Ucal \rangle_{\te{cpt}}$ and $\pc{\Ucal}{\Ycal}$ are endowed with a $\tdop$-action. Clearly one has that $\Ucal \subset  \langle \Ucal \rangle_{\te{cpt}}$ is open and quasi-compact. 
\end{remark}

\begin{lemma}\label{inv:nbd}
For any point $z\in \cpt{\Ucal}$, there is a $\tdop$-invariant open neighborhood $\Wcal \subset \cpt{\Ucal}$ which contains $z$.  
\end{lemma}

\begin{proof}
Recall that $\cpt{\Ucal} = \varprojlim \overline{\Ucal}'$, with $\{\overline{\Ucal}' \}$ an inverse system of $\Ucal$-admissible blow ups. Set $z'=\pi'(z)$, where $\pi': \cpt{\Ucal} \to \overline{\Ucal}'$ is the canonical projection. Let $\Wcal \subset \cpt{\Ucal}$ be an open neighborhood of $z$, then we know from \cite[Proposition 0.2.2.9]{fujiwara-kato14} that there exists an element $\overline{\Ucal}'$ in the inverse system and an open neighborhood $\Vcal \subset \overline{\Ucal}'$ such that $\Wcal ={\pi'}^ {-1}(\Vcal)$. Furthermore, we may assume that $\Ucal \subset \Vcal$. By taking $\Vcal':=\bigcup t\cdot \Vcal$ with $t\in T^\circ(K):=\{t\in T(K)| |t|=1\}$, if necessary, we get an open $\tdop$-invariant subset of $\overline{\Ucal}'$. Therefore, the open $\Wcal'={\pi'}^{-1}(\Vcal')$ is a $\tdop$-invariant neighbourhood of $z$. 
\end{proof}

\begin{prop}
Let $\Ucal_1\subset  \Ucal_2$ be two $\tdop$-invariant open affine subsets of $\Ycal$, then we have a $\tdop$-equivariant open embedding
\[ \pc{\Ucal_1}{\Ycal} \hra \pc{\Ucal_2}{\Ycal}\]
extending the inclusion of $\Ucal_1$ into $\Ucal_2$. 
\end{prop}

\begin{proof}
Let $\overline{\Ucal_2}$ be a compactification of $\Ucal_2$ over $\kcirc$. By taking the closure of $\Ucal_1$ in $\overline{\Ucal_2}$, we get a compactification of $\Ucal_1$ over $\kcirc$ as well. We have a canonial morphism  $\cpt{\Ucal_1}\xrightarrow{\varphi} \cpt{\Ucal_2}$ extending the inclusion  $\Ucal_1 \hra \Ucal_2$.  
Now let $x,y \in \cpt{\Ucal_1}\backslash \Ucal_1$ be such that $\varphi(x)=\varphi(y)\in \cpt{\Ucal_2}$ and therefore have the same projection on $\overline{\Ucal_2}'\backslash \Ucal_1$, where $\overline{\Ucal_2}'\to \overline{\Ucal_2}$ is a $\Ucal_2$-admissible blow up.  After replacing $x$ and $y$ by generizations of $x$ and $y$ if needed, we may assume that $\pi(x)= \pi(y) \in \overline{\Ucal_1}'\cap \Ucal_2$, where $\pi \colon \cpt{\Ucal_1} \to \overline{\Ucal_1}'$ is the projection and $\overline{\Ucal_1}' \to \overline{\Ucal_1}$ is a $\Ucal_1$-admissible blow up. By Proposition \ref{blow up extension} this implies that $x, y \in \zr{\Ucal_2}{\Ucal_1}= \zr{\Ycal}{\Ucal_1} \subset \cpt{\Ucal_1}$. Therefore the induced map $\pc{\Ucal_1}{\Ycal} \to \cpt{\Ucal_2}$ is an open immersion. 
\end{proof}

Now consider an open covering $\{ \Ucal_i \}$ of the $\tdop$-toric variety $\Ycal$. Given two elements of this covering $\Ucal_i, \Ucal_j$ we have the canonical inclusions $\Ucal_i\cap \Ucal_j \hra \Ucal_i$ and  $\Ucal_i\cap \Ucal_j \hra \Ucal_j$, which by the previous proposition give rise to open embeddings $\pc{\Ucal_i \cap \Ucal_j}{\Ycal}\hra \pc{\Ucal_i}{\Ycal}$ and $\pc{\Ucal_i \cap \Ucal_j}{\Ycal}\hra \pc{\Ucal_j}{\Ycal}$. Therefore we have well defined maps $ \amalg \pc{\Ucal_i \cap \Ucal_j}{\Ycal} \mathrel{\mathop{\rightrightarrows}^{\mathrm{p}}_{\mathrm{q}}} \amalg \pc{\Ucal_i}{\Ycal}$ coming from those inclusions. We define the \emph{$\tdop$-invariant Zariski--Riemann compactification} $\cpt{\Ycal}$  of $\Ycal$ as the cokernel of these maps in the category of locally ringed spaces. We get the exact sequence
\[ \amalg \pc{\Ucal_i \cap \Ucal_j}{\Ycal} \mathrel{\mathop{\rightrightarrows}^{\mathrm{p}}_{\mathrm{q}}} \amalg \pc{\Ucal_i}{\Ycal} \to \cpt{\Ycal}.\]
By construction $\cpt{\Ycal}$ is $\tdop$-equivariant. 

To show that this space is algebraic, we proceed as follows.

\begin{prop} \label{algebraic}
Given a point $z\in \cpt{\Ycal}$, there exists a dense open immersion over $\kcirc$ $\Ycal \hra \Ycal_z$ of $\tdop$-toric varieties over $\kcirc$ such that $\zr{\Ycal_z}{\Ycal}$ contains the point z. 
\end{prop}

\begin{proof}
This follows from \cite[Lemma F.3.3]{fujiwara-kato14} by noting that from the proof of Lemma \ref{inv:nbd}, we may assume that there is a $\tdop$-invariant open subset $\Vcal \subset \overline{\Ucal}'$ which contains $\Ucal$ and such that $\pi^{-1}(\Vcal)=\Wcal$, where $\pi: \cpt{\Ucal} \to \overline{\Ucal}'$ is the canonical projection and $\overline{\Ucal}'\to \overline{\Ucal}$ is a $\Ucal$-admissible blow up. Hence, $\Ycal_z$ is obtained by gluing $\Ycal$ and $\Vcal$ along $\Ucal$ as in Proposition \ref{gluing}.
\end{proof}

Finally, by using the quasi-compactness of the $\tdop$-invariant Zariski--Riemann space we can prove our main result.

\begin{proof}[Proof of Theorem 1]
For any point $z\in \cpt{\Ycal}$ construct $\Ycal_z$ as in Proposition \ref{algebraic}. As $\cpt{\Ycal}$ is quasi-compact, there exist finitely many points $z_i \in \cpt{\Ycal}$ such that $\{ \zr{\Ycal_{z_i}}{\Ycal}\}$ is an open covering of $\cpt{\Ycal}$.  Then by applying Proposition \ref{gluing}, we get a $\tdop$-toric variety ${\Ycal}_\te{cpt}$ over $\kcirc$ which contains $\Ycal$ and satisfies $\zr{{\Ycal}_\te{cpt}}{\Ycal}=\cpt{\Ycal}$. It follows from \cite[Corollary F.2.13]{fujiwara-kato14} that it is proper over $\kcirc$.
\end{proof}

\bibliographystyle{plain}
\bibliography{soto}

{\small Alejandro Soto,  Institut f\"ur Mathematik, Goethe-Universit\"at Frankfurt, Robert-Mayer-Strasse 8, D-60325 Frankfurt am Main, soto@math.uni-frankfurt.de }

\end{document}